\documentclass[12pt]{article}
\usepackage{graphicx}
\usepackage{amsmath}
\usepackage{graphicx}
\usepackage{amsfonts}
\usepackage{amssymb}
\usepackage{mathrsfs}
\usepackage{esint}
\usepackage{xcolor}
\usepackage{cancel}
\usepackage{hyperref}
\usepackage{enumitem}
\usepackage{mathrsfs}
\usepackage{amsmath, amssymb ,amsthm, amsfonts, amsgen}
\DeclareGraphicsExtensions{.eps,.bmp,.jpg,.pdf,.mps,.png,.gif}
\numberwithin{equation}{section} 
\newcommand{\NN}{\mathbb{N}}

\newcommand{\RR}{\mathbb{R}}

\newtheorem{definition}{\sc Definition}[section]
\newtheorem{teo}{\sc Theorem}[section]

\newtheorem{lemma}{\sc Lemma }[section]

\usepackage[titletoc]{appendix}
\newcommand{\lqqd}{\hfill $\square$}      %end proof

\begin{document}
%%%%%%%%%%%%%%%%%%%%%%%%%%%%%%%%%%
\title{Kirchhoff-type equations involving the Fractional $(p,q)-$Laplacian}

%%%%%%%%%%%%%%%%%%%%%%%%%%%%%%%%%%

\author{Lisbeth Carrero\footnote{Instituto de Ciencias de la Ingenier\'ia, Universidad de O'Higgins, Rancagua, Chile. E-mail: lisbeth.carrero@postdoc.uoh.cl}\;  and Pedro Hern\'andez-Llanos\footnote{Instituto de Ciencias de la Ingenier\'ia, Universidad de O'Higgins, Rancagua, Chile. E-mail: pedro.hernandez@postdoc.uoh.cl}
}

\date{}

\maketitle{}

%%%%%%%%%%%%%%%%%%%%%%%%%%%%%%%%%%
\begin{abstract}
%%%%%%%%%%%%%%%%%%%%%%%%%%%%%%%%%%
In this paper, we study the existence and nonexistence of solutions for the following Kirchhoff-type fractional $(p\text{-}q)$-Laplacian problem:

\begin{equation*}
\begin{cases}
 M\left([u]^p_{p,s_1}\right)(-\Delta)^{s_1}_p u + M\left([u]^q_{q,s_2}\right)(-\Delta)^{s_2}_q u = \lambda\big[a(x)|u|^{p-2}u + b(x)|u|^{q-2}u\big] + h(x), & \text{in } \Omega, \\
u = 0, & \text{on } \mathbb{R}^N \setminus \Omega,
\end{cases}
\end{equation*}

where $\Omega \subset \mathbb{R}^N$ ($N \geq 1$) is a bounded domain with smooth boundary, $0 < s_1 < s_2 < 1$, and $s_1 p < N$. We assume $1 < q \leq p < \theta p < p^{*}_{s_1} := \dfrac{Np}{N - s_1 p}$, and $\lambda \in \mathbb{R}$. The functions $a(x), b(x)$, and $h(x)$ are non-negative, with $a, b \in L^\infty(\Omega)$ and $h \in L^q(\Omega)$.

Using variational methods, we establish the existence of at least two weak solutions. The first solution is obtained via the direct minimization of the associated energy functional, and the second is obtained by applying the Mountain Pass Theorem. We also prove a nonexistence result for small values of the parameter $\lambda > 0$.

{\it Key words} Fractional $(p,q)-$Laplacian operator, Kirchhoff type problem, mountain pass theorem.

{\rm 2010 AMS Subject Classification: 35A15, 35J20, 35J60, 35R11}
\end{abstract}

%%%%%%%%%%%%%%%%%%%%%%%%%%%%%%%%%%
\section{Introduction}
%%%%%%%%%%%%%

In this paper we study the existence of solutions for the following $(p-q)-$Laplacian equation Kirchhoff type 
\begin{equation}\label{eq:3}
\begin{cases}
 M\left([u]^p_{p,s_1}\right)(-\Delta)^{s_1}_pu+M\left([u]^q_{q,s_2}\right)(-\Delta)^{s_2}_qu =  \lambda[a(x)|u|^{p-2}u+b(x)|u|^{q-2}u]+h(x),\, &\text{in}\,\,\Omega\\
u(x)=0,\, &\text{on}\,\, \RR^N\setminus\Omega,
\end{cases}
\end{equation}
 where \(\Omega \subset \mathbb{R}^N\) (\(N \geq 1\)) is a bounded domain with smooth boundary in \(\mathbb{R}^N\), with \(0 < s_1 < s_2 < 1 \) such that $s_1p<N$ and \(\lambda \in \mathbb{R}\). Here, $1< q \leq p < \theta p < p^{*}_{s_1}= \frac{Np}{N-s_1p}$. The functions $a(x), b(x)$, and $h(x)$ are assumed to be non-negative, with $a, b \in L^\infty(\Omega)$ and $h \in L^q(\Omega)$. The operator $(-\Delta)^{s}_p$ denotes the fractional $p$-Laplacian, which arises in the description of various phenomena in applied science, such as the phase transition \cite{SireValdinoci2009}, fractional quantum mechanics \cite{Laskin2000}, as well as in continuum mechanics, finance, biology, physics, population dynamics, and game theory, among others \cite{MetzlerKlafter2004,Viswanathan1996} .  For $s\in (0,1)$ it can be defined via the Fourier multiplier $|\xi|^{2s}$. Moreover, for suitable functions like $u\in L^{\infty}(\mathbb{R}^{N})\cap C^{2s+ \beta}_{loc}(\mathbb{R}^{N})$ it can also be defined by the formula
\begin{equation}\label{fracLapl}
(-\Delta)^s_p u(x) := {C_{N, s,p} \mathrm{P.V.} \int_{\mathbb{R}^N} \frac{|u(x) - u(z)|^{p-2}(u(x) - u(z))}{|x - z|^{N + ps}}dz}, \quad x\in\RR.
\end{equation}
Here, $C_{N, s,p} > 0$ is a well-known normalizing constant, and P.V. stands for the principal value.

Consider the classical Kirchhoff equation
\begin{equation*}
-\left( a+b\int_{\RR^N}|\nabla u(x)|^2\, dx\right) \Delta u+u=f(u), \quad \text{in} \quad \RR^N    
\end{equation*}
which has been extensively studied over the past decade. This equation corresponds to the stationary version of the original Kirchhoff model, initially introduced by Kirchhoff in 1883 \cite{Kirchhoff1897} as a generalization of the classical D’Alembert wave equation, which describes free vibrations of elastic strings. Subsequently, Fiscella and Valdinoci \cite{FiscellaValdinoci2014} extended this framework by proposing a fractional stationary Kirchhoff equation in bounded regular domains of $\RR^N$, capturing nonlocal phenomena.

\vspace{0.5 cm}
In 2015, P. Pucci et al. \cite{PucciXiangZhang2015} were interested in obtain multiple solutions for nonhomogeneous fractional $p-$Laplacian equations of Schr\"odinger–Kirchhoff type 
\begin{equation}\label{eq:pucci2015}   
M\left([u]^p_{s,p}\right)(-\Delta)^{s}_{p}u=-V(x)|u|^{p-2}u+f_q(x,u)+h(x)\quad \text{in}\quad \RR^N
\end{equation}
 where $0<s<1<p<\infty$. They study the case where $f_q(x,u)$ satisfy some growth condition and the case where $f_q(x,u)=|u|^{q-2}u$ with $V(x)\equiv 1$.\\

Following extensive work on the fractional p-Laplacian, researchers have shifted their focus to a new class of nonlocal operators, the fractional $(p,q)$-Laplacian equations. \\

In 2019, Alves et. al. \cite{Alvesetal2019} established existence results for the following class of problems involving the fractional $(p-q)$-Laplacian operators:
\begin{equation*}
(-\Delta)^s_{p}u+(-\Delta)^{s}_{q}u+V(\varepsilon x)\left(|u|^{p-2}u+|u|^{q-2}u\right)=f(u),\quad\text{in}\quad \RR^N, 
\end{equation*}
where $\varepsilon>0$ is a small parameter, $s\in (0,1)$, $1<p<q<\frac{N}{s}$, $V:\RR^N\to\RR$ is a continuous function verifying
\begin{equation*}
V_{\infty}:=\displaystyle\liminf_{|x|\to \infty}V(x)>V_0:=\inf_{\RR^N}V(x)>0,    
\end{equation*} and the nonlinearity $f\in C^1(\RR,\RR)$ satisfies standard conditions compatible with variational methods. In the same year, Bhakta and Mukherjee \cite{BhaktaMukherjee} discuss the existence of multiple nontrivial solutions of $(p-q)-$fractional Laplacian equations which involve concave-critical type nonlinearities and existence of
nonnegative solutions when nonlinearities is of convex-critical type. They consider equations of the type
\begin{equation*}
\begin{cases}
(-\Delta)^{s_1}_p u+(-\Delta)^{s_2}_q u=\theta V(x)|u|^{r-2}u+|u|^{p^{*-2}_{s_1}}u+\lambda f(x,u),&\quad \text{in}\,\Omega,\\
u=0 &\quad \text{in}\, \RR^{N}\setminus\Omega,
\end{cases}
\end{equation*}
where $\Omega\subset \RR^N$ is smooth, bounded domain, $\lambda,\theta>0$, $0<s_2<s_1<1$, $1<r<q<p<\frac{N}{s_1}$ and $p^{*}_s=\frac{Np}{N-sp}$ for any $s\in(0,1)$. The functions $f$ and $V$ satisfy certain assumptions. \\

Also, Goel et. al. \cite{Goeletal2019} study the existence and multiplicity of solutions of the following fractional problem
\begin{equation}\label{eq:pqab}
\begin{cases}
(-\Delta)^{s_1}_p u+\beta(-\Delta)^{s_2}_q u=\lambda a(x)|u|^{\delta-2}u+b(x)|u|^{r-2}u,&\quad \text{in}\,\Omega,\\
u=0 &\quad \text{in}\, \RR^{N}\setminus\Omega,
\end{cases}
\end{equation}
where $\Omega$ is a bounded domain in $\RR^N$ with smooth boundary $1<\delta\leq q\leq p<r\leq p^{*}_{s_1}$, with $p^{*}_{s_1}=\displaystyle\frac{np}{n-ps_1}$, $0<s_2<s_1<1$, $n>ps_1$, $\lambda, \beta>0$ are parameters and $a\in L^{\frac{r}{r-\delta}}(\Omega)$ and $b\in L^{\infty}(\Omega)$ are sign changing functions.

Biswas and Sk \cite{Biswas2024} analyze the case for fractional $(p,q)-$Laplace equation in (\ref{eq:pqab}), considering constant coefficients by taking $a(x)=\alpha$, $b(x)=\beta$ and $\lambda=1$ , i.e.
\begin{equation*}\label{eq:biswasSk}
\begin{cases}
(-\Delta)^{s_1}_{p}u+(-\Delta)^{s_2}_qu=\alpha |u|^{q-2}u+\beta|u|^{q-2}u,&\quad \text{in}\,\Omega,\\
u=0 &\quad \text{in}\, \RR^{N}\setminus\Omega,
\end{cases}
\end{equation*}
where $\Omega\subset\RR^d$ is a bounded open set and $\alpha,\beta\in\RR$.\\

There are many works on this topic. We are only mentioning a few here. Readers can refer to the references for further details. In addition, we would include \cite{PucciXiangZhang2015} and its citations as another interesting contribution. \\

In this work, we focus on existence and nonexistence results for the Kirchhoff-type equation involving the $(p-q)$-fractional Laplacian operator, as given in (\ref{eq:3}). We assume certain conditions on the Kirchhoff function, which are described below:
\begin{enumerate}[label=$(\pmb{M\arabic*})$, itemsep=0.25em]
    \item\label{item:M:1} Let \( M: \mathbb{R}^{+} \to \mathbb{R}^{+} \) be an increasing and continuous function.  There exists $m_0>0$ such that 
    \begin{equation*}
        M(t)\geq m_0, \quad \text{for all } t\in\RR^{+}.
    \end{equation*}
\item\label{item:M:2} Let $\mathscr{M}(t):=\displaystyle\int_{0}^{t}M(\tau)\,d\tau. $ There exists $\theta \in \left[1,\frac{p^{*}_{s_1}}{p}\right)$ such that $\theta \mathscr{M}  \geq M(t) t$ for all $t\in\RR^{+}$.
\end{enumerate}
\vspace{0.2cm}
A typical example of $M$ is given by $M(t)=c+dt^m$ with $m>0,c>0,d\geq 0$ for all $t\geq 0.$ Another example of $M$  under the assumption \ref{item:M:2}, is cases where \( M \) is not monotone. For instance, consider  
\[
 M(t) = (1 + t)^k + (1 + t)^{-1}, \quad \text{for } t \in \mathbb{R}_0^{+},
 \]
 with \( 0 < k < 1 \). In this case, \( \theta = k + 1 \), and condition \ref{item:M:2} is satisfied provided that \( k \in (0,1) \) is sufficiently small so that \( \theta = k + 1 < \dfrac{p_{s_1}^{*}}{p}\).\\

\vspace{0.5 cm}

Our main result is stated as follows:
\begin{teo}\label{thm:1}
Assume that conditions ~\ref{item:M:1}-~\ref{item:M:2} are satisfied, and that  $a(x), b(x)$, and $h(x)$ are non-negative functions with $a, b \in L^\infty(\Omega)$, and $h(x) \in L^q(\Omega)$. Then, 
\begin{enumerate}
    \item[(i)] If $h \equiv 0$, then there exists a constant $\lambda^* > 0$ such that, for all $0 < \lambda < \lambda^{*}$, the problem admits no nontrivial solution.
    \item[(ii)]  If $\|h\|_{L^q(\Omega)}$ is sufficiently small, then there exists a constant $\lambda^{**}> 0$ such that, for all $0 < \lambda < \lambda^{**}$, problem \eqref{eq:3} admits at least two non-negative solutions.

\end{enumerate}

\end{teo}
Our approach to proving Theorem~\ref{thm:1} is based on variational methods. We begin by proving part $(i)$, which concerns the nonexistence of nontrivial solutions for small values of $\lambda$. The key idea is to test the equation with the function $u$, and using condition~\ref{item:M:1} along with Hölder's inequality, derive the desired result. In part $(ii)$, we establish the existence of two solutions: one is obtained by minimizing the associated energy functional, and the other via the Mountain Pass Theorem.

%\textcolor{red}{Why is important study this problem? What's new with respect other already existing models}
%The paper is organized as follows: Sect. \ref{sect2} introduces the definition and some properties needed in order to prove existence of solutions and finally we proved our main theorem.

\vspace{0.5 cm}
 
\section{Existence of solutions}\label{sect2}
%%%%%%%%%%%%%%%%%%%%%%%%%%%%%%%%%%
In order to apply variational methods to our problem, we work within the fractional Sobolev space  $W^{s,p}_0(\Omega)$ defined as follows
\begin{equation*}
W^{s,p}_0(\Omega):=\{u\in L^p(\RR^N): u=0\,\text{in}\,\RR^N\setminus\Omega\quad\text{and}\quad [u]_{s,p}<\infty\},\,
\end{equation*}
where $ 0<s<1<p$ and
\begin{equation}\label{Gagliardoseminorm}
[u]_{s,p}=\left(\iint_{\RR^{2N}}\frac{|u(x)-u(y)|^p}{|x-y|^{N+ps}}\,dx\,dy\right)^{1/p}    
\end{equation}
is the so-called Gagliardo (semi)norm of $u$ (see \cite{NezzaPalatucci}).

We now state a useful embedding result between two fractional Sobolev spaces, which will be used later in this work.

\begin{lemma}\label{lemma1:oct15}
Let $0<s_1<s_2<1<q\leq p$ and $\Omega$ be a smooth bounded domain in $\RR^N$, where $s_1p<N$. Then $W^{s_1,p}_0(\Omega)\subset W^{s_2, p}_0(\Omega)$ and there exists $C=C(|\Omega|,N, p,q, s_1, s_2)>0$ such that
\begin{equation*}
 ||u||_{W^{s_2,q}_0(\Omega)}\leq C||u||_{W^{s_1,p}_0(\Omega)}.   
\end{equation*}
\end{lemma}
The proof of the previuos Lemma can be found in \cite[Lemma 2]{BhaktaMukherjee}.
\vspace{0.5 cm}

We now introduce the following auxiliary problem, which plays a key role, as we will prove the existence results for this problem.
\begin{equation}\label{eq:probaux}
\begin{cases}
 M\left([u]^p_{p,s_1}\right)(-\Delta)^{s_1}_pu+M\left([u]^q_{q,s_2}\right)(-\Delta)^{s_2}_qu\\= \lambda a(x)|u_{+}|^{p-2}u_{+}
 +\lambda b(x)|u_{+}|^{q-2}u_{+} +h(x),\, &\text{in}\,\,\Omega\\
u(x)=0,\, &\text{on}\,\, \RR^N\setminus\Omega.
\end{cases}
\end{equation}
The following lemma shows us that the solutions to our auxiliary problem are non-negative, which allows us to conclude that the solutions to our original problem are non-negative.
\begin{lemma}
    Assume that \( a \), \( b \), and \( h \) are non-negative functions. Then, any solution \( u \in W_0^{s_1,p}(\Omega) \) to the problem \eqref{eq:probaux} is non-negative. In particular, \( u \in W_0^{s_1,p}(\Omega) \), is a solution to problem \eqref{eq:3}.
\end{lemma}
\begin{proof}
Suppose that $u \in W_0^{s_1,p}(\Omega)$ is a solution to \eqref{eq:probaux}. Testing the equation with $u_{-}:=\min \{0,u\}$ and the following inequality,
\begin{equation*}
|x - y|^{p - 2}(x - y)(x_{-} - y_{-}) \geq |x_{-} - y_{-}|^p, \quad \forall p\geq 1,
\end{equation*}
we obtain
\begin{align*}
m_0[u_{-}]^p_{s_1} +m_0 [u_{-}]^q_{s_2} -\int_{\Omega} h(x)u_-\leq&M([u]^p_{s_1})\int_{\Omega}(-\Delta)^{s_1}_p u u_- + M([u]^q_{s_2})\int_{\Omega}(-\Delta)^{s_2}_q uu_- \\
&-\int_{\Omega} \lambda\left(a(x)|u_{+}|^{p-2}u_+u_- - b(x)|u_{+}|^{q-2}u_+u_-\right)\\
&- \int_{\Omega} h(x)u_-=0.
\end{align*}
It follows that,
\begin{align*}
m_0  ||u_||^{p}_{W^{s_1,p}_0}\leq m_0[u_{-}]^p_{s_1} +m_0 [u_{-}]^q_{s_2}  -\int_{\Omega} h(x)u_{-}  \leq 0.  
\end{align*}
Thus, we conclude that $u_- = 0$. Hence, $u \geq 0$ is a solution to the problem \eqref{eq:3}.

%\fin     
\end{proof}
 We introduce the following energy functional   $J:W^{s_1,p}_0(\Omega) \to \RR $ defined by 
\begin{align}
J(u)&\nonumber=\displaystyle\frac{1}{p}\mathscr{M}\left([u]^p_{s_1,p}\right)+\displaystyle\frac{1}{q}\mathscr{M}\left([u]^q_{s_2,q}\right)-\displaystyle\frac{\lambda}{p}\int_{\Omega}a(x)|u_{+}|^p\,dx-\displaystyle\frac{\lambda}{q}\int_{\Omega}b(x)|u_{+}|^q\,dx\\
&\label{functenergy:Oct08}\quad-\displaystyle\int_{\Omega}h(x)u(x)\,dx,
\end{align}
which is associated to \eqref{eq:probaux}. 

\vspace{0.5 cm}

\begin{lemma}\label{wlscfunctionalJ:march15}
  $J\in C^{1}(W^{s_1,p}_0(\Omega))$ is weakly lower semi-continuous in $W^{s_1,p}_0(\Omega)$. 
\end{lemma}

\begin{proof}

\vspace{0.5 cm}
First, we prove that $J\in C^{1}(W^{s_1,p}_0(\Omega))$. It is standard to see that $J$ is Frech\^{e}t differentiable in $W^{s_1,p}_0(\Omega)$, with derivative
\begin{align*}
 \langle J'(u), \varphi \rangle &= 
M\left([u]^p_{s_1,p}\right)\iint_{\Omega\times \Omega}   
\frac{|u(x) - u(y)|^{p-2}}{|x - y|^{N+s_1 p}}(u(x) - u(y)) (\varphi(x) - \varphi(y)) \, dx \, dy\\
&\quad + 
M\left([u]^q_{s_2,q}\right)\iint_{\Omega\times\Omega}  \frac{|u(x) - u(y)|^{q-2}}{|x - y|^{N+s_2 q}}  (u(x) - u(y)) (\varphi(x) - \varphi(y)) \, dx \, dy\\
&\quad- \lambda \int_{\Omega} \left[ 
a(x) |u_{+}|^{p-2} + b(x) |u_{+}|^{q-2} 
\right] u_{+} \varphi \, dx - \int_{\Omega} h\varphi\,dx, \quad \forall \varphi\in W^{s_1,p}_0(\Omega).
\end{align*}
It remains to be proved that $J'$ is continuous. Let $\{u_n\}_{n\in\NN}\subset W^{s_1,p}_{0}(\Omega)$ and $u\in W^{s_1,p}_0$ satisfy $u_n\rightarrow u$ in $W^{s_1,p}(\Omega)$ as $n\to\infty$. Without loss of generality, we assume that $u_n\rightarrow u$ a.e. in $\RR^N$. Then, the sequence
\begin{equation*}
   \left\{\displaystyle\frac{|u_n(x)-u_n(y)|^{p-2}\left(u_n(x)-u_n(y)\right)}{|x-y|^{\frac{N+s_1p}{p'}}}\right\}_{n=1}^{\infty}\quad\text{is bounded in}\, L^{p'}(\RR^{2N})
\end{equation*}
and
\begin{equation*}
\left\{\frac{|u_n(x)-u_n(y)|^{q-2}\left(u_n(x)-u_n(y)\right)}{|x-y|^{\frac{N+s_2q}{p'}}}\right\}_{n=1}^{\infty}\quad\text{is bounded in}\, L^{q'}(\RR^{2N}) 
\end{equation*}
as well a.e. in $\RR^{2N}$, where the pair $p$ and $p'$ are conjugated index, i.e. $\frac{1}{p}+\frac{1}{p'}=1$. 
\begin{align*}
    \mathcal{U}_n(x,y):=\displaystyle\frac{|u_n(x)-u_n(y)|^{p-2}\left(u_n(x)-u_n(y)\right)}{|x-y|^{\frac{N+s_1p}{p'}}}\rightarrow \mathcal{U}(x,y):=\displaystyle\frac{|u(x)-u(y)|^{p-2}\left(u(x)-u(y)\right)}{|x-y|^{\frac{N+s_1p}{p'}}}
\end{align*}
and
\begin{align*}
    \mathcal{V}_n(x,y):=\displaystyle\frac{|u_n(x)-u_n(y)|^{q-2}\left(u_n(x)-u_n(y)\right)}{|x-y|^{\frac{N+s_2q}{q'}}}\rightarrow \mathcal{V}(x,y):=\displaystyle\frac{|u(x)-u(y)|^{q-2}\left(u(x)-u(y)\right)}{|x-y|^{\frac{N+s_2q}{q'}}}.
\end{align*}
Thus, the Brezis-Lieb lemma (see \cite{brezislieb}) implies
\begin{align}
&\nonumber\displaystyle\lim_{n\to \infty}\iint_{\RR^{2N}}|\mathcal{U}_n(x,y)-\mathcal{U}(x,y)|^{p'}\,dx\,dy\\
&\label{eq1:march19}=\displaystyle\lim_{n\to \infty}\iint_{\RR^{2N}}\left(\frac{|u_n(x)-u_n(y)|^p}{|x-y|^{N+ps_1}}-\frac{|u(x)-u(y)|^p}{|x-y|^{N+ps_1}}\right)\,dx\,dy
\end{align}
and
\begin{align}
&\nonumber\displaystyle\lim_{n\to \infty}\iint_{\RR^{2N}}|\mathcal{V}_n(x,y)-\mathcal{V}(x,y)|^{q'}\,dx\,dy\\
&\label{eq2:march19}=\displaystyle\lim_{n\to \infty}\iint_{\RR^{2N}}\left(\frac{|u_n(x)-u_n(y)|^q}{|x-y|^{N+qs_2}}-\frac{|u(x)-u(y)|^q}{|x-y|^{N+qs_2}}\right)\,dx\,dy.
\end{align}
Moreover, the continuity of $M$ implies that
\begin{equation}\label{eq3:march19}
\displaystyle \lim_{n\to\infty} M([u_n]^p_{s_1,p})=M([u]^p_{s_1,p}), 
\end{equation}
and
\begin{equation}\label{eq4:march19}
\displaystyle \lim_{n\to\infty} M([u_n]^q_{s_2,q})=M([u]^q_{s_2,q}).
\end{equation}
From (\ref{eq1:march19}) and (\ref{eq2:march19}) it follows that
\begin{equation}\label{eq5:march19}
   \displaystyle\lim_{n\to \infty}\iint_{\RR^{2N}}|\mathcal{U}_n(x,y)-\mathcal{U}(x,y)|^{p'}\,dx\,dy=0 
\end{equation}
and
\begin{equation}\label{eq6:march19}
   \displaystyle\lim_{n\to \infty}\iint_{\RR^{2N}}|\mathcal{V}_n(x,y)-\mathcal{V}(x,y)|^{q'}\,dx\,dy=0.
\end{equation}
Similarly,

\begin{equation}\label{eq7:march19}
   \displaystyle\lim_{n\to \infty}\int_{\RR^{N}}a(x)||u_n(x)|^{p-2}u_n(x)-|u(x)|^{p-2}u(x)|^{p'}\,dx=0, 
\end{equation}
and
\begin{equation}\label{eq8:march19}
   \displaystyle\lim_{n\to \infty}\int_{\RR^{N}}b(x)||u_n(x)|^{q-2}u_n(x)-|u(x)|^{q-2}u(x)|^{q'}\,dx=0.
\end{equation}

Combining (\ref{eq3:march19})-(\ref{eq7:march19}) and (\ref{eq8:march19}) with H\"older's inequality, we have
\begin{align*}
||J'(u_n)-J'(u)||_{\left(W^{s_1,p}_0(\Omega)\right)'}=\sup_{\varphi\in (W^{s_1,p}_0(\Omega)), ||\varphi||_{W^{s_1,p}_0(\Omega)}=1}\left|\langle J'(u_n)-J'(u), \varphi\rangle\right|\rightarrow 0,   
\end{align*}
as $n\to\infty$. 
Finally, notice that the maps $v\mapsto [v]^p_{s_1,p}$ and $v\mapsto [v]^q_{s_2,q}$ are lower semi-continuous in the weak topology of $W^{s_1,p}(\RR^N)$ and $W^{s_2,q}(\RR^N)$ respectively and $\mathcal{M}$ is nondecreasing and continuous on $\RR^{+}_0$, so that $v\mapsto\mathcal{M}([v]^p_{s_1,p})$ and  $v\mapsto \mathcal{M}([v]^q_{s_2,q})$ are semi-continuous in the weak topology of $W^{s_1,p}(\RR^N)$ and $W^{s_2,q}(\RR^N)$ respectively. Indeed, we can define $\psi: W^{s_1,p}(\RR^N)\mapsto \RR$ and $\varphi: W^{s_2,q}(\RR^N)\mapsto \RR$ as follows
\begin{align*}
    \psi(v)&=\displaystyle\iint_{\RR^{2N}}|v(x)-v(y)|^p|x-y|^{-N-s_1p}\,dx\,dy,\\    \varphi(u)&=\displaystyle\iint_{\RR^{2N}}|u(x)-u(y)|^q|x-y|^{-N-s_2q}\,dx\,dy.
\end{align*}
To analyze the differentiability of the functionals $\psi$ and $\varphi$, we observe that both are defined in terms of the Gagliardo seminorm, which is known to be continuously differentiable $C^1$ in fractional Sobolev spaces. Besides, both $\psi$ and $\varphi$ are convex functionals in $W^{s_1,p}(\mathbb{R}^N)$ and $W^{s_2,q}(\mathbb{R}^N)$, respectively.  Indeed, since the function $f(\xi) = |\xi|^p$ is convex for $p \geq 1$, the integrand
\begin{equation*}
\displaystyle \frac{|v(x)-v(y)|^p}{|x-y|^{N+s_1p}}    
\end{equation*}
is convex in $v$ (pointwise in the variable $(x,y)$). Integration preserves convexity; hence, $\psi$ is convex as a functional on $W^{s_1,p}(\mathbb{R}^N)$. A similar argument applies to $\varphi$ on $W^{s_2,q}(\mathbb{R}^N)$. By Corollary 3.8 in \cite{Brezis2011}, we obtain
\begin{align*}
    \psi(v)&\leq \liminf_{n\to \infty}\psi (v_n) 
\end{align*}
and
\begin{align*}
    \varphi(v)&\leq \liminf_{n\to \infty}\psi (v_n). 
\end{align*}
Let's continue to prove the weak lower semicontinuity of the last terms in (\ref{functenergy:Oct08}), by hypothe\-sis $a(x)$ and $b(x)$ are positive functions belonging to $L^\infty(\Omega)$, and $h(x) \in L^q(\Omega)$. By the compact embedding $W_0^{s_1,p}(\Omega) \hookrightarrow L^p(\Omega)$ and $W_0^{s_2,q}(\Omega) \hookrightarrow L^q(\Omega)$, along with the Dominated Convergence Theorem and Hölder’s inequality, we obtain
\begin{equation*}
\lim_{n \to \infty} \int_{\Omega} a(x) |(u_n)_+(x)|^p \, dx = \int_{\Omega} a(x) |u_+(x)|^p \, dx,
\end{equation*}
\begin{equation*}
\lim_{n \to \infty} \int_{\Omega} b(x) |(u_n)_+(x)|^q \, dx = \int_{\Omega} b(x) |u_+(x)|^q \, dx,
\end{equation*}
and
\begin{equation*}
\lim_{n \to \infty} \int_{\Omega} h(x) u_n(x) \, dx = \int_{\Omega} h(x) u(x) \, dx.    
\end{equation*}
Since the sum of weakly lower semicontinuous functionals remains weakly lower semicontinuous, it follows that the functional $J(u)$ is weakly lower semicontinuous on the space $W_0^{s_1,p}(\Omega)$.
\end{proof}
\vspace{0.25 cm}
Let \( \lambda^a_{1,p} >0 \) be the first eigenvalue of the weighted fractional \( p \)-Laplacian problem:
\begin{equation*}
\begin{cases}
(-\Delta)^{s_1}_p u = \lambda a(x) |u|^{p-2}u & \text{in } \Omega, \\
u = 0 & \text{in } \mathbb{R}^N \setminus \Omega,
\end{cases}
\end{equation*}
where  $a(x)$ is a non-negative weight function and $u \in W^{s_1,p}_0(\Omega)$.

Then, define 
\begin{equation*}
\lambda_{0} := \frac{m_0}{\theta} \lambda^a_{1,p}.
\end{equation*}

\begin{lemma}\label{lemma2:oct15}
 Assume that $\lambda< \lambda_0$, then the energy functional $J$ defined by (\ref{functenergy:Oct08}) is coercive.
\end{lemma}

\begin{proof}
Let be $u\in W^{s_1,p}_0(\Omega)$ and suppose that assumptions~\ref{item:M:1}–\ref{item:M:2} are satisfied. Then by Poincar\'e inequality and H\"older's inequality, it follows that 
\begin{align*}
    J(u)&=\displaystyle\frac{1}{p}\mathscr{M}\left([u]^p_{s_1,p}\right)+\displaystyle\frac{1}{q}\mathscr{M}\left([u]^q_{s_2,q}\right)-\displaystyle\frac{\lambda}{p}\int_{\Omega}a(x)|u_{+}|^p\,dx-\displaystyle\frac{\lambda}{q}\int_{\Omega}b(x)|u_{+}|^q\,dx\\
&\quad-\displaystyle\int_{\Omega}h(x)u(x)\,dx,\\
&\geq \displaystyle\frac{m_0}{\theta p}||u||^{p}_{W^{s_1,p}_0(\Omega)}+\displaystyle\frac{m_0}{\theta q}||u||^{q}_{W^{s_2,p}_{0}(\Omega)}-\displaystyle\frac{\lambda}{p}\int_{\Omega}a(x)|u|^p\,dx-\displaystyle\frac{\lambda}{q}\int_{\Omega}b(x)|u|^q\,dx\\
&\quad-C_1||h||_{L^q(\Omega)}||u||_{W^{s_1,p}_0(\Omega)}\\
&\geq \left(\displaystyle\frac{m_0}{\theta p}- \displaystyle\frac{\lambda}{p\lambda^a_{1,p}}\right)||u||^{p}_{W^{s_1,p}_0(\Omega)}-\frac{\lambda}{q}C_2||b||_{L^{\infty}}||u||^q_{W^{s_2,q}_0(\Omega)}- C_1||h||_{L^q(\Omega)}||u||_{W^{s_1,p}_0(\Omega)}.
\end{align*}
Since  
\[
\frac{m_0}{\theta p} - \frac{\lambda}{p\lambda^a_{1,p}} > 0,
\]  
we conclude that the functional \( J \) is coercive and bounded from below.
\end{proof}
\vspace{0.5 cm}
In preparation for proving our main result, we present concepts and results that will be instrumental in our analysis. We begin by defining a Palais–Smale sequence, which plays a central role in variational methods and the Mountain Pass Theorem, a fundamental tool for establishing the existence of nontrivial solutions to nonlinear equations. 
\vspace{0.5 cm}
\begin{definition}\label{Palais-Smaledef} (Palais-Smale condition). Let $\{u_n\}_{n=1}^{\infty}$ be a sequence. We say that $\{u_n\}_{n=1}^{\infty}$ is a Palais–Smale sequence at level $c$, denoted $(PS)_c$, for the functional $J$, if
\begin{itemize}
    \item[(i)] $J(u_n)\to c$ as $n\to \infty$ and
    \item[(ii)] $J'(u_n)\to 0$ in $(W_0^{s,p}(\Omega))'$as $n\to \infty$.
\end{itemize}
\end{definition}

\begin{teo}\label{teopasodemontana}
(Mountain Pass Theorem, cf. \cite{AmbrosettiRabino,ServadeiValdinochi,Willem}). Let \( X \) be a real Banach space and let \( \Phi \in C^1(X, \mathbb{R}) \). Suppose there exist constants \( \beta > 0 \), \( \rho > 0 \), and a point \( x_1 \in X \setminus B_\rho(0) \) such that
\begin{itemize}
    \item[(i)] \( \Phi(u) \geq \beta \) for all \( u \in X \) with \( \|u\|_X = \rho \);
    \item[(ii)] \( \Phi(x_1) < \beta \), and \( \Phi(0) = 0 \).
\end{itemize}
Then, the functional \( \Phi \) possesses a Palais–Smale sequence at the level
\begin{equation*}
c := \inf_{\gamma \in \Gamma} \max_{t \in [0,1]} \Phi(\gamma(t)),    
\end{equation*}
where
\begin{equation*}
\Gamma := \left\{ \gamma \in C([0,1], X) \ : \ \gamma(0) = 0,\ \gamma(1) = x_1 \right\}.    
\end{equation*}
\end{teo}

\begin{teo}\label{teogeo} Assume \( 0 < s_1 < s_2 \leq 1 < q < p \), then there exists \( \lambda_{00} > 0 \) such that:
\begin{enumerate}
\item [(1)] There exist \( \delta > 0 \) and \( r > 0 \) such that for all \( \lambda \in (0, \lambda_{00}) \), \( J(u) \geq \delta \) with \( \| u \|_{W^{s_1, p}_0(\Omega)} = r \).
\item [(2)] There exists $u_0\in W^{s_1, p}_0(\Omega)$ with $||u_0||_{W^{s_1, p}_0(\Omega)}>r$ and $J(u_0)<0$.
\end{enumerate}
\end{teo}
\begin{proof}
Let $u\in W^{s_1, p}_0(\Omega)$ and $r > 0$ such that $\| u \|_{W^{s_1, p}_0(\Omega)} = r$. Then, using ~\ref{item:M:1}–\ref{item:M:2} and Lemma \ref{lemma1:oct15} it is follows
\begin{align*}
J(u)&\geq \displaystyle\frac{1}{\theta p}m_0||u||^p_{W^{s_1,p}_0}+\frac{1}{\theta q}m_0||u||^q_{W^{s_2,q}_0}\\
&\quad-\frac{\lambda C_1}{p}||u||_{W^{s_1,p}_0}^{p}-\frac{\lambda}{q}C_2||u||^q_{W^{s_2,q}_0}-C_3||h||_{L^q}||u||_{W^{s_1,p}_0}\\
&\geq \displaystyle\frac{1}{\theta p}m_0||u||^p_{W^{s_1,p}_0}-\displaystyle\frac{\lambda}{p}C_1||u||^p_{W^{s_1,p}_0}-\frac{\lambda}{q}C_2||u||^q_{W^{s_1,p}_0}-C_3||h||_{L^q}||u||_{W^{s_1,p}_0}\\
&=||u||^q_{W^{s_1,p}_0}\left(\displaystyle\frac{1}{\theta p}m_0||u||^{p-q}_{W^{s_1,p}_0}-\frac{\lambda}{p}C_1||u||^{p-q}_{W^{s_1,p}_0}-\displaystyle\frac{\lambda}{q}C_2\right)-C_3||h||_{L^{q}}||u||_{W^{s_1,p}_0},
\end{align*}
where $C_1=C_1(||a||_{\infty}),C_2= C_2(||b||_{\infty})$ and $C_3$ it is the constant that comes from the Poincare inequality. \\
Now, we shall ensure that
\begin{equation*}
\displaystyle\frac{1}{\theta p}m_0||u||^{p-q}_{W^{s_1,p}_0}-\frac{\lambda}{p}C_1||u||^{p-q}_{W^{s_1,p}_0}-\displaystyle\frac{\lambda}{q}C_2>0,
\end{equation*}
if we take
\begin{equation*}
\lambda_{00}:=\displaystyle\frac{m_0}{\theta p}\displaystyle\frac{t^{p-q}_0}{\displaystyle\frac{C_2}{q}+\displaystyle\frac{C_1}{p}t_0^{p-q}}, \quad t_0>0.
\end{equation*}
It is hold that
\begin{equation*}
\delta:= \displaystyle\frac{1}{\theta p}m_0t_0^{p-q}-\frac{\lambda_{00}}{p}C_1t_0^{p-q}-\displaystyle\frac{\lambda_{00}}{q}C_2>0,
\end{equation*}
So, as $||h||_{L^{q}}$ is enough small, then for every $\lambda\in (0, \lambda_{00})$ and $t_0=r$ we obtain $J(u)\geq \delta$. 
\vspace{0.5 cm}
Now, let's prove that $(2)$ is also verified. In indeed, 
\begin{align*}
J(tu)&=\displaystyle \frac{1}{p}\mathscr{M}(t^p[u]^p_{s_1})+\frac{1}{q}\mathscr{M}(t^q[u]^q_{s_2})-\frac{\lambda}{p}t^p\int_{\Omega}a(x)|u_{+}|^p\,dx-\frac{\lambda}{p}t^q \int_{\Omega}b(x)|u_{+}|^q\,dx-t\int_{\Omega}h(x) u\,dx.
\end{align*}
By  ~\ref{item:M:1}–\ref{item:M:2} and Poincar\'e inequality we have 
\begin{align*}
J(tu)&\geq \displaystyle\frac{1}{\theta p}m_0 t^p ||u||^p_{W^{s_1,p}_0(\Omega)}+\frac{1}{\theta q}m_0 t^q||u||^q_{W^{s_2,q}_0(\Omega)}-\frac{\lambda}{p}t^p C_1||u||^p_{W^{s_1,p}_0(\Omega)}\\
&\quad-\displaystyle\frac{\lambda}{p}t^{q}C_2||u||^q_{W^{s_2,q}_0(\Omega)}-C_3t||h||_{L^q}||u||_{W^{s_1,p}_0(\Omega)}\\
&\geq \displaystyle\frac{1}{\theta p}m_0 t^p||u||^p_{W^{s_1,p}_0(\Omega)}-\lambda\frac{C_1}{p}t^p ||u||^p_{W_0^{s_1,p}(\Omega)}-\frac{\lambda}{q}t^{q}C_2||u||^q_{W^{s_1,p}_0(\Omega)}-C_3t||h||_{L^q}||u||_{W^{s_1,p}_0(\Omega)}.\\
& \geq \displaystyle\left(\frac{m_0}{\theta p} -\dfrac{m_0C_1}{\theta p}\right) t^p||u||^p_{W^{s_1,p}_0(\Omega)}-\frac{\lambda}{q}t^{q}C_2||u||^q_{W^{s_1,p}_0(\Omega)}-C_3t||h||_{L^q}||u||_{W^{s_1,p}_0(\Omega)}.
\end{align*}
where $C_1 := C_1(\lambda, \|a\|_{L^{\infty}}, \text{Poincaré inequality constant}) > 1$, $C_2 := C_2(\|b\|_{L^{\infty}})$, and $C_3 := C_3(\text{The Poincaré inequality constant})$. Then, since $p > q$, we have $J(tu)\to-\infty$ as $t\to\infty$.
\end{proof}

The following lemma will be fundamental in establishing that the $(PS)_{c}$ sequence of the functional $J$ defined by (\ref{functenergy:Oct08}) satisfies the $(PS)_c$-condition. These results are well-known vector inequalities for $1<p<\infty$ called Simon's inequality.

\begin{lemma}\label{lemmaA1} 
\begin{itemize} Given $\xi,\eta\in \RR^N$
    \item[(a)] For $1<p\leq 2$ the following inequality hold
\begin{equation}\label{eqauxiliar:a1}
   \displaystyle\frac{|\xi-\eta|^p}{\left(|\xi|^{2}+|\eta|^{\frac{2-p}{2}}\right)}\leq C_1(p)\left(|\xi|^{p-2}\xi-|\eta|^{p-2}\eta\right)^{\frac{p}{2}}(\xi-\eta)^{\frac{p}{2}}.
\end{equation}

\item[(b)]In the other hand, for $2<p<\infty$, we have the following inequality
\begin{equation}\label{eqauxiliar:a2}
|\xi-\eta|^p\leq C_2(p)\left(|\xi|^{p-2}\xi-|\eta|^{p-2}\eta\right)(\xi-\eta).
\end{equation}
    
\end{itemize} 
    
\end{lemma}

\vspace{0.5 cm}

\begin{lemma}\label{PS-conditionlemma}
 Let $s_1 p<N$ and $0\leq \lambda\neq \lambda^{a,\Omega}_{1,p}$. If $\{u_n\}_{n=1}^{\infty}$ is a bounded $(PS)_c$ sequence of the functional $J$ defined by (\ref{functenergy:Oct08}), then the functional satisfies $(PS)_c$ condition.
\end{lemma}

\begin{proof}
 Let $\{u_n\}_{n=1}^{\infty}$ be bounded a $(PS)_c$ sequence according to Definition \ref{Palais-Smaledef}, then
 \begin{equation}\label{eq1:31012025}
J(u_n)\to c\quad\text{and}\quad ||J'(u_n)||_{\left(W^{s_1,p}_0(\Omega)\right)'}\to 0\,\,\text{as}\,\, n\to\infty.     
 \end{equation}
and there exits a subsequence, still denoted by $\{u_n\}_{n=1}^{\infty}$, such that
\begin{align}
 &\label{eq2:conver}u_n\rightharpoonup u\quad\text{in}\,\, W^{s_1,p}_0(\Omega),\, u_n\longrightarrow u\,\,\text{a.e. in}\,\, \RR^N\\
 &\nonumber u_n\longrightarrow u\quad\text{in}\, L^{\gamma}(\Omega)\,\,\text{for}\,\,q\leq \gamma\leq p^{*}_{s_1}.
\end{align}
To begin with, we prove $u_n\rightarrow u$ in $W^{s_1,p}_0(\Omega)$. Let us fix $\varphi\in W^{s_1, p}_0(\Omega)$ and denote by
 \begin{align}
A_p(u,\varphi) &= \iint_{\mathbb{R}^{2N}} \frac{|u(x) - u(y)|^{p-2} (u(x) - u(y))}{|x - y|^{N + s_1 p}} \, (\varphi(x) - \varphi(y)) \, dx \, dy, \\
B_q(u,\varphi) &=\iint_{\mathbb{R}^{2N}} \frac{|u(x) - u(y)|^{q-2} (u(x) - u(y))}{|x - y|^{N + s_2 q}} \, (\varphi(x) - \varphi(y)) \, dx \, dy.
\end{align}

where $A_p(u,\varphi)$ and $B_q(u,\varphi)$ are the linear functions with respect to $\varphi$. Clearly, by virtue of H\"older's inequality then $(A_p+B_q)(u,\varphi)$ is also continuous and
\begin{align}
 |A_p(u,\varphi)+B_q(u,\varphi)&\label{eq4:31012025}|\leq |A_p(u,\varphi)|+|B_q(u,\varphi)|\\
 &\nonumber \leq [u]^{p-1}_{s_1,p}[\varphi]_{s_1,p} +[u]^{q-1}_{s_2,q,}[\varphi]_{s_2,q}\\
 &\leq \left(||u||^{p-1}_{W^{s_1,p}_0(\Omega)}+||u||^{q-1}_{W^{s_2,q}_0(\Omega)}\right)||\varphi||_{W^{s_1,p}_0(\Omega)}.
\end{align}

By \ref{eq2:conver} and the fact that $\{M([u]^{p}_{s_1,p}) -M([u]^{p}_{s_1,p})\}_n$ and $\{M([u]^{q}_{s_2,q}) -M([u]^{q}_{s_2,q})\}_n$ are bounded in $\RR$, we obtain  
\begin{equation*}
\displaystyle\lim_{n\to\infty}[M([u]^{p}_{s_1,p}) -M([u]^{p}_{s_1,p}) ]A_p(u, u_n-u)=0.
\end{equation*}
\begin{equation}\label{eq1:abcero}
 \displaystyle\lim_{n\to\infty}[M([u]^{q}_{s_2,q}) -M([u]^{q}_{s_2,q}) ]B_q(u,u_n-u)=0.   
\end{equation}
It is clear that $(J'(u_n) - J'(u), u_n - u) \to 0$ as $n\to \infty$. Therefore,
\begin{align}
o_n(1) &= (J'(u_n) - J'(u), u_n - u) \label{ineqned:03022025} \nonumber\\
&= M\bigl([u_n]_{s_1,p}^p\bigr) A_p(u_n,u_n - u) + M\bigl([u_n]_{s_2,q}^q\bigr) B_q(u_n, u_n - u)\nonumber \\
&\quad - \bigl(M\bigl([u]_{s_1,p}^p\bigr) A_p(u, u_n - u) + M\bigl([u]_{s_2,q}^q\bigr) B_q(u, u_n - u)\bigr) - \lambda \Lambda_n,
\end{align}
where
\begin{equation*}
\Lambda_n = \int_{\Omega} \Bigl[ a(x)\bigl(|u_n^+|^{p-2}u_n^+ - |u^+|^{p-2}u^+\bigr) + b(x)\bigl(|u_n^+|^{q-2}u_n^+ - |u^+|^{q-2}u^+\bigr) + \lambda^{-1}h(x) \Bigr] (u_n - u) \, dx.
\end{equation*}
From (\ref{eq2:conver}) it follows that \(\Lambda_n \to 0\) as $n \to \infty$. Hence, by (\ref{eq1:abcero}) and the above, we have 
\begin{equation*}
\lim_{n \to \infty} M\bigl([u_n]_{s_1,p}^p\bigr)\bigl[A_p(u_n,u_n - u)-A_p(u,u_n - u)\bigr] + M\bigl([u_n]_{s_2,q}^q\bigr)\bigl[B_q(u_n,u_n - u)-B_q(u,u_n - u)\bigr] = 0.
\end{equation*}
Since, by \ref{item:M:1}, we have 
\begin{equation*}
M\bigl([u_n]_{s_1,p}^p\bigr)\bigl[A_p(u_n,u_n - u)-A_p(u,u_n - u)\bigr]\geq 0    
\end{equation*}
and
\begin{equation*}
M\bigl([u_n]_{s_2,q}^q\bigr)\bigl[B_q(u_n,u_n - u)-B_q(u,u_n - u)\bigr]\geq 0.
\end{equation*}
Therefore, it follows that
\begin{equation}\label{eq:limab}
\begin{cases}
\displaystyle\lim_{n \to \infty} \bigl[A_p(u_n,u_n - u)-A_p(u,u_n - u)\bigr]=0, & \\
 \displaystyle\lim_{n \to \infty} \bigl[B_q(u_n,u_n - u)-B_q(u,u_n - u)\bigr] = 0. &
\end{cases}
\end{equation}

Let us now consider two cases: First, suppose that $p > q > 2$ and second one $1< p\leq 2$.

\underline{\textit{Case $p > q > 2$}}: By Lemma \ref{lemmaA1}(b), and taking $\mathcal{O} = \Omega \times \Omega$, we obtain that
 \begin{align}
     [u_n-u]^p_{s_1,p,\mathcal{O}}&\nonumber=\displaystyle\iint_{\mathcal{O}}\frac{|u_n(x)-u_n(y)-(u(x)-u(y))|^p}{|x-y|^{N+s_1p}}\,dx\,dy\\
     &\nonumber\leq C_2(p)\displaystyle\int_{\mathcal{O}}[|u_n(x)-u_n(y)|^{p-2}(u_n(x)-u_n(y))-|u(x)-u(y)|^{p-2}(u(x)-u(y))]\\
     &\nonumber\times [u_n(x)-u_n(y)-(u(x)-u(y))]|x-y|^{-N-s_1p}\,dx\,dy\\
     &\label{eq3:03022025}\leq C_2(p)[A_p(u_n,u_n-u)-A_p(u, u_n-u)]
 \end{align}
 and similarly that
\begin{align}
     [u_n-u]^q_{s_2,q, \mathcal{O}} &\nonumber=\displaystyle\iint_{\mathcal{O}}\frac{|u_n(x)-u_n(y)-(u(x)-u(y))|^q}{|x-y|^{N+s_1p}}\,dx\,dy\\
     &\nonumber\leq C_2(q)\displaystyle\int_{\mathcal{O}}[|u_n(x)-u_n(y)|^{q-2}(u_n(x)-u_n(y))-|u(x)-u(y)|^{q-2}(u(x)-u(y))]\\
     &\nonumber\times [u_n(x)-u_n(y)-(u(x)-u(y))]|x-y|^{-N-s_2q}\,dx\,dy\\
     &\label{eq4:03022025}\leq C_2(q)[B_q(u_n, u_n-u)-B_q(u,u_n-u)].
 \end{align}
 Let $C_0=\min\{C^{-1}_2(p), C^{-1}_{2}(q)\}$. 

Thus, by \ref{eq:limab}, we can conclude that $u_n\to u$ in $W^{s_1,p}_{0}(\Omega)$ as $n\to \infty$.

\vspace{0.5 cm}

\underline{\textit{Case $1< p\leq 2$}}: By the hypothesis, there exists $M>0$ such that $[u_n]_{s_1, p}\leq M$. Then, by Lemma \ref{lemmaA1}(a) and the H\"older's inequality, we obtain that
\begin{align}
 [u_n-u]^{p}_{s_1,p,\mathcal{O}}&\nonumber\leq C_3(p)[A_p(u_n-u,u_n)-A_p(u_n-u,u)]^{\frac{p}{2}}\left([u_n]^{\frac{p(2-p)}{2}}_{s_1,p,\mathcal{O}}+[u]^{\frac{p(2-p)}{2}}_{s_1,p,\mathcal{O}}\right)\\
 &\label{eq1:04022025}\leq C_4(p)[A_p(u_n-u,u_n)-A_p(u_n-u,u)]^{\frac{p}{2}}.
\end{align}
Analogously, for $1<q<2$, we have
\begin{align}
 [u_n-u]^{q}_{s_2,q,\mathcal{O}}&\nonumber\leq C_5(p)[B_q(u_n-u,u_n)-B_q(u_n-u,u)]^{\frac{q}{2}}\left([u_n]^{\frac{q(2-q)}{2}}_{s_2,q,\mathcal{O}}+[u]^{\frac{q(2-q)}{2}}_{s_2,q,\mathcal{O}}\right)\\
 &\label{eq2:04022025}\leq C_6(q)[B_q(u_n-u,u_n)-B_q(u_n-u,u)]^{\frac{q}{2}}.
\end{align}
Combining inequalities (\ref{eq1:04022025}) and (\ref{eq2:04022025}), we obtain
\begin{align*}
 &A_p(u_n-u,u_n)-A_p(u_n-u,u)+B_q(u_n-u,u_n)-B_q(u_n-u,u)\\
 &\quad \geq C_1\left([u_n-u]^2_{s_1,p,\mathcal{O}}+[u_n-u]^2_{s_2,q,\mathcal{O}}\right),   
\end{align*}
 with some $C_1=\min\{C_4(p)^{-1}, C_6(p)^{-1}\}>0$. This proves $u_n\rightarrow u$ strongly in $W^{s_1,p}_0(\Omega)$ as $n\to \infty$. Therefore, $J$ satisfies the $(PS)_c$ condition in $W^{s_1,p}_0(\Omega)$.

\end{proof}
\begin{lemma}\label{sbound}
The sequence $\{u_n\}_{n=1}^{\infty}$ is bounded in $W^{s_1,p}_0(\Omega)$.
\end{lemma}

\begin{proof}
Let $\{u_n\}_{n=1}^{\infty}$ be a $(PS)_{c}$ sequence (see Definition \ref{Palais-Smaledef}). Then there exists a constant $C > 0$ such that

$$
|J(u_n)| \leq C \quad \text{and} \quad |\langle J'(u_n), u_n \rangle| \leq C \|u_n\|_{W_0^{s_1, p}(\Omega)}.
$$
\begin{align*}
C+C||u||_{W^{s_1,p}_0(\Omega)}&\geq J(u_n)-\displaystyle\frac{1}{p_{s_1}^{*}}\langle J'(u_n), u_n\rangle\\
&=\displaystyle\frac{1}{p}\mathscr{M}\left([u_n]^p_{s_1,p}\right)-\frac{1}{p_{s_1}^{*}}M\left([u_n]^p_{s_1,p}\right)[u_n]^p_{s_1,p}\\
&\quad+\displaystyle\frac{1}{q}\mathscr{M}\left([u_n]^{q}_{s_2,q}\right)-\frac{1}{p_{s_1}^{*}}M\left([u_n]^q_{s_2,q}\right)[u_n]^q_{s_2,q}\\
&\quad -\displaystyle\lambda\left(\frac{1}{p}-\frac{1}{p_{s_1}^{*}}\right)\int_{\Omega}a(x)|u_n^+|^p-\lambda\left(\frac{1}{q}-\frac{1}{p_{s_1}^{*}}\right)\int_{\Omega}b(x)|u_n^+|^q-\frac{p_{s_1}^{*}- 1}{p_{s_1}^{*}}\int_{\Omega}hu_n\\
&\geq\displaystyle\left(\frac{1}{p}-\frac{1}{p_{s_1}^{*}}\right)M\left([u_n]^p_{s_1,p}\right)[u_n]^{p}_{s_1,p}+\displaystyle\left(\frac{1}{q}-\frac{1}{p_{s_1}^{*}}\right)M\left([u_n]^q_{s_2,q}\right)[u_n]^{q}_{s_2,q}\\
&\quad -\lambda C_1||u_n||^p_{W^{s_1,p}_{0}(\Omega)}-C_2||u_n||^q_{W^{s_2,q}_{0}(\Omega)}-C_3||u_n||_{W_0^{s_1,q}(\Omega)}\\
&\geq \displaystyle\left(\frac{1}{\theta p}-\frac{1}{p_{s_1}^{*}}\right)m_0||u_n||^p_{W^{s_1,p}_0(\Omega)}-\lambda C_1||u_n||^p_{W^{s_1,p}_0(\Omega)}\\
&-C_2||u_n||^q_{W^{s_1,p}_0(\Omega)}-C_3||u_n||_{W^{s_1,p}_0(\Omega)}\\
&\geq \displaystyle ||u_n||^{p}_{W^{s_1,p}_0(\Omega)}\left(\left(\frac{1}{p}-\frac{1}{p_{s_1}^{*}}\right)m_0-\lambda C_1\right)- C_2||u_n||^{q}_{W^{s_1,p}_0(\Omega)}-C_3||u_n||_{W^{s_1,p}_0(\Omega)},
\end{align*}
where  $C_1=(||a||_{L^{\infty}(\Omega)}),$ $C_2=(\lambda, ||b||_{L^{\infty}(\Omega)})$ and $C_1=(||h||_{L^{p}(\Omega)}).$
Then, taking $$\lambda < \dfrac{\left(\frac{1}{\theta p}-\frac{1}{p_{s_1}^{*}}\right)m_0 }{C_1}:=\lambda_{000},$$
we can conclude that the sequence $\{u_n\}_{n=1}^{\infty}$ is bounded in $W^{s_1,p}_0(\Omega)$.
\end{proof}

\vspace{0.5 cm}

\subsection*{Proof of Theorem \ref{thm:1}:}
First, we prove part $(i)$. Let $\lambda>0 $ be such that the problem \ref{eq:probaux} admits a nontrivial solution $u\in W^{s_1,p}_0(\Omega)$. By testing the equation with $u$, we obtain
\begin{equation*}
M([u]^p_{s_1})\int_{\Omega}(-\Delta)^{s_1}_p u u + M([u]^q_{s_2})\int_{\Omega}(-\Delta)^{s_2}_q uu \\
=\int_{\Omega} \lambda\left(a(x)|u_{+}|^{p-2}u_+u + b(x)|u_{+}|^{q-2}u_+u\right).
\end{equation*}
Now, using ~\ref{item:M:1} and Hölder's inequality, we have
\begin{equation*}
 \displaystyle m_0||u||^p_{W^{s_1,p}_0}+m_0||u||^q_{W^{s_2,q}_0} \leq \lambda C\left(||u||^p_{W^{s_1,p}_0}+||u||^q_{W^{s_2,q}_0}\right).
\end{equation*}
It follows that there exists $\lambda^*>0$ such that problem \ref{eq:probaux} does not admit any nontrivial solution for all $\lambda \in (0, \lambda^*)$. To prove the remaining part, i.e. $(ii)$, let $\lambda^{*} := \min \{ \lambda_0, \lambda_{00}, \lambda_{000} \}$. For every $\lambda < \lambda^{*}$, it follows from Lemma \ref{wlscfunctionalJ:march15}  and Lemma \ref{lemma2:oct15} that the functional $J$  achieves a global minimum at $u_1 \in W^{s_1, p}_0(\Omega)$, and this function $u_1$ is a solution to equation \ref{eq:probaux}.
To obtain a second solution, observe that for every $\lambda < \lambda^{*}$, Theorem \ref{teogeo}, ensures that the hypotheses of Theorem \ref{teopasodemontana}  are satisfied. Therefore, there exists a sequence $\{u_n\}_{n=1}^{\infty} \subset W^{s_1,p}_0(\Omega)$ such that 

\[
J(u_n) \to c \quad \text{and} \quad J'(u_n) \to 0, \quad \text{as}\quad n\to \infty.
\]
Then, by Lemmas \ref{PS-conditionlemma} and \ref{sbound} we have that $J$ admits a critical point $u_2 \in W^{s_1,p}_0(\Omega)$. Moreover, 
\begin{equation*}
J(u_2)=c >0 =J(0),
\end{equation*}
which implies that $u_2\ne 0$. Since 
$u_1$ is a global minimizer of $J$, we have $J(u_1)<c=J(u_2)$, and therefore $u_2\ne u_1$. 
\lqqd

\section*{Acknowledgments}
L.C. was partially supported by ANID Chile through the Postdoctoral Fondecyt Grant No. 3240062.  P. H.-L., was partially supported by ANID Chile through the Postdoctoral Fondecyt Grant No. 3230202 and also was funded partially by Instituto de Ciencias Matem\'aticas (ICMAT) at Autonomous University of Madrid through program ``Puentes Matem\'aticos con Latinoam\'erica" 2025. P. H.-L. thankful to Carlos Mora-Corral for the kind and generous hosting at the ICMAT and gratefully acknowledges the generous hospitality provided by the Department of Mathematics and Mathematical Statistics at Ume\r{a} University, parti\-cularly by Sebastian Throm, \r{A}ke Br\"annstr\"om and Konrad Abramowicz.


\begin{thebibliography}{99}


%%%%new 2023%%%%

\bibitem{Alvesetal2019} {\sc C.O. Alves, V. Ambrosio and  T. Isernia}, {\rm  Existence, multiplicity and concentration for a class of fractional $p\& q$ Laplacian problems in $\RR^N$,} {\it Commun. Pure Appl. Anal.} {\bf 18} (2019), no. 4, 2009--2045.

\bibitem{AmbrosettiRabino} {\sc A. Ambrosetti and  P. Rabinowitz}, {\rm  Dual variational methods in critical point theory and applications,} {\it J. Funct. Anal.} {\bf 14} (1973) 349--381.

\bibitem{AmbrosioIsernia2018} {\sc V. Ambrosio and  T. Isernia}, {\rm On a fractional, $p\&q$ Laplacian problem with critical Sobolev-Hardy exponents,} {\it Mediterr. J. Math.} {\bf 15}, no. 219 (2018) 493--516.

\bibitem{Brezis2011} {\sc H. Br\'ezis}, {\rm Functional Analysis,Sobolev Spaces and Partial Differential Equations.} {\it Universitext, Springer}, New York (2011)

\bibitem{brezislieb} {\sc H. Brezis and E. Lieb}, {\rm A relation between pointwise convergence of functions and convergence of functionals.} {\it Proc. Amer. Math. Soc.}, 88 (1983) no. 3, 486--490.

\bibitem{Biswas2024} {\sc N. Biswas and F. Sk}, {\rm On generalized eigenvalue problems of fractional $(p,q)-$Laplace operator with two parameters.} {\it Proc. R. Soc. Edinb.}, 1-46, (2024), 

\bibitem{BhaktaMukherjee} {\sc Y. M. Bhakta and D. Mukherjee}, {\rm Multiplicity results for $(p, q)$ fractional elliptic equations involving critical nonlinearities}, {\it Adv. Diﬀerential Equations.} {\bf 24} (2019) no.3--4, 185-228.  

\bibitem{CaponiPucci2016} {\sc M. Caponi and P. Pucci}, {\rm Existence theorems for entire solutions of stationary Kirchhoff fractional p-Laplacian equations}, {\it Ann. Mat. Pura Appl.} {\bf 195} (2019) no.6, 2099--2129.  

\bibitem{ChenFeng2024} {\sc W. Chen, D. Feng}, {\rm Critical Fractional $(p,q)-$Kirchhoff Type Problem with a
Generalized Choquard Nonlinearity and Magnetic Field}, {\it Bull. Malays. Math. Sci. Soc.} 47:33 (2024).  

\bibitem{FariaMiyagakiMotreanu2014} {\sc L. Faria, O. Miyagaki, D. Motreanu}, {\rm Comparison and positive solutions for problems with $(p, q)-$Laplacian and convection term,} {\it Proc. Edinb. Math. Soc.} Vol.1, {\bf 57}, (2014) 687--698.

\bibitem{FiscellaValdinoci2014}
{\sc A. Fiscella, E. Valdinoci}, 
{\rm A critical Kirchhoff type problem involving a nonlocal operator}, 
{\it Nonlinear Anal.} {\bf 94} (2014), 156--170.

\bibitem{Goeletal2019} {\sc D. Goel, D. Kumar and K. Sreenadh}, {\rm Regularity and multiplicity results for fractional $(p,q)-$Laplacian equations}, {\it Comun. Contemp. Math.}(2019)  
 
\bibitem{JiaGaoZhang} {\sc Y. Jia, Y. Gao, G. Zhang}, {\rm Solutions for the Kirchhoff type equations with fractional Laplacian}, {\it J. Appl. Anal. Comput.} {\bf 6} (2020) no.6  2704--2710.

\bibitem{Kirchhoff1897}
{\sc G. Kirchhoff}, 
{\rm Vorlesungen über Mechanik}, 
Vol.~1, B.G. Teubner, Leipzig, 1897.


\bibitem{Laskin2000}
{\sc N. Laskin}, 
{\rm Fractional quantum mechanics and Lévy path integrals}, 
{\it Phys. Lett. A} {\bf 268} (4--6) (2000), 298--305.

\bibitem{MaranoPapageorgiou2013} {\sc S. A. Marano, N. S. Papageorgiou}, {\rm Constant-sign and nodal solutions of coercive
$(p, q)-$Laplacian problems,} {\it Nonlinear Anal. }{\bf 77}, (2013) 118–129.

\bibitem{MetzlerKlafter2004}
{\sc R. Metzler, J. Klafter}, 
{\rm The restaurant at the end of the random walk: Recent developments in the description of anomalous transport by fractional dynamics}, 
{\it J. Phys. A: Math. Gen.} {\bf 37} (2004), R161- R208.


\bibitem{MotreanuTanaka2016} {\sc D. Motreanu, M. Tanaka}, {\rm On a positive solution for $(p,q)-$Laplace equation with indefinite weight,} {\it Minimax Theory Appl.} Vol.1, {\bf 1}, (2016) 1--20.


\bibitem{NezzaPalatucci} {\sc E. D. Nezza, G. Palatucci, E. Valdinocci}, {\rm Hitchhicker's guide to the fractional Sobolev spaces,} {\it Bull. Sci. math.} 136 (2012) 521--573.

 
\bibitem{Nguyenandvo} {\sc T-H. Nguyen, H-H. Vo}, {\rm Principal eigenvalue and positive solutions for fractional Laplace operator in quantum field theory}, arXiv:2006.03233v1 [math.AP], (2020).

\bibitem{PucciXiangZhang2015} {\sc P. Pucci, M.Q. Xiang, B.L. Zhang}, {\rm Multiple solutions for nonhomogeneous Schr\"odinger-Kirchhoff type equations involving the fractional $p-$Laplacian in $\RR^N$,} {\it Calc. Var.} {\bf 54} (2015) 2785--2806.


\bibitem{ServadeiValdinochi} {\sc R. Servadei and E. Valdinocci}, {\rm  Mountain Pass solutions for non-local elliptic operators,} {\it J. Math. Anal. Appl.} {\bf 389} (2012) no. 2, 887--898

\bibitem{Sidiropoulos2010} {\sc N. E. Sidiropoulos}, {\rm  Existence of solutions to indefinite quasilinear elliptic problems of $P-Q-$Laplacian type,} {\it Electron. J. Diﬀer. Equ.} {\bf 2010(162)} (2010) 23 p.

\bibitem{SireValdinoci2009} 
{\sc Y. Sire, E. Valdinoci}, 
{\rm Fractional Laplacian phase transitions and boundary reactions: A geometric inequality and a symmetry result}, 
{\it J. Funct. Anal.} {\bf 256} (6) (2009), 1842-1864.

\bibitem{Willem} {\sc M. Willem}, {\rm  Minimax Theorems,} {\it Birkh\"auser} (1996).

\bibitem{WuYang2009} {\sc M. Wu, Z. Yang}, {\rm A class of $p-q-$Laplacian type equation with potentials eigenvalue
problem in $\RR^N$,} {\it Bound. Value Probl.} ID 185319 (2009) 19 p.


\bibitem{XiangZhangFerrara2015} {\sc M.Q. Xiang, B.L. Zhang and M. Ferrara}, {\rm Existence of solutions for Kirchhoff type problem involving the non-local fractional p-Laplacian,} {\it J. Math. Anal. Appl.} {\bf 424}, 1021-1041 (2015).

\bibitem{YinYang2011} {\sc H. Yin, Z. Yang}, {\rm  A class of $p-q-$Laplacian type equation with noncave-convex
nonlinearities in bounded domain,} {\it J. Math. Anal. Appl.} {\bf 382} (2011), 843–855.


\bibitem{YinYang2012} {\sc H. Yin, Z. Yang}, {\rm  Multiplicity of positive solutions to a $p-q-$Laplacian equations
involving critical nonlinearity,} {\it Nonlinear Anal.} {\bf 75} (2012) , 3021–3035.

\bibitem{Viswanathan1996}
{\sc G. Viswanathan, V. Afanasyev, E. Murphy, P. Prince, H. Stanley}, 
{\rm Lévy flight search patterns of wandering albatrosses}, 
{\it Nature} {\bf 381} (1996), 413--415.


\end{thebibliography}
\end{document}